\documentclass[10pt]{amsart}
\usepackage[english]{babel}
\usepackage{amssymb}
\usepackage{amsmath}
\usepackage{srcltx}
\usepackage{lscape}

\newcommand{\can}{\overline{\phantom{x}}}

\newtheorem{dummy}{Dummy}

\newtheorem{lemma}[dummy]{Lemma}
\newtheorem{theorem}[dummy]{Theorem}
\newtheorem{proposition}[dummy]{Proposition}
\newtheorem{corollary}[dummy]{Corollary}

\theoremstyle{definition}

\newtheorem{example}[dummy]{Example}

\newcommand{\ignore}[1]{}

\author{C. Brown}
\author{S. Pumpl\"un}

\email{Christian.Brown@nottingham.ac.uk; susanne.pumpluen@nottingham.ac.uk}
\address{School of Mathematical Sciences\\
University of Nottingham\\ University Park\\ Nottingham NG7 2RD\\
United Kingdom }

\keywords{Skew polynomial ring, skew polynomials, Ore polynomials, automorphisms, nonassociative algebras.}

\subjclass[2010]{Primary: 17A35; Secondary: 17A60, 17A36, 16S36}

\begin{document}

\title[The automorphisms of Petit's algebras]
{The automorphisms of Petit's algebras}

\begin{abstract}
Let $\sigma$ be an automorphism of a field $K$ with fixed field $F$.
We study the automorphisms of nonassociative unital algebras which are canonical generalizations
of the associative quotient algebras $K[t;\sigma]/fK[t;\sigma]$ obtained when the twisted polynomial $f\in K[t;\sigma]$ is invariant,
and were first defined by Petit.
We compute all their automorphisms  if $\sigma$ commutes with all automorphisms in ${\rm Aut}_F(K)$ and $n\geq m-1$,
where $n$ is the order of $\sigma$ and $m$ the degree of $f$,
and obtain  partial results for $n<m-1$.
In the case where $K/F$ is a finite Galois field extension, we obtain more detailed information on the structure of
 the automorphism groups of these nonassociative unital algebras over $F$.
We also briefly investigate when two such algebras are isomorphic.
\end{abstract}

\maketitle

\section*{Introduction}

Let $D$ be a division algebra, $\sigma$  an injective endomorphism of $D$,
$\delta$ a left $\sigma$-derivation and $R = D[t;\sigma,\delta]$ a skew polynomial ring (for instance, c.f. \cite[\S~3.4]{N1}).
For an invariant skew polynomial $f\in R$,  i.e. when the ideal $Rf$ is a two-sided principal ideal, the quotient
algebra $R/Rf$ appears in classical constructions of associative central simple algebras, usually
employing an irreducible $f\in R$ to get examples of  division algebras, e.g. see \cite{J96}.

In 1967, Petit \cite{P66,P68} introduced a class of unital nonassociative algebras
$S_f$, which canonically generalize the quotient
algebras $R/Rf$  obtained when factoring out an invariant $f\in R$ of degree $m$.
The algebra
$S_f=D[t;\sigma,\delta]/D[t;\sigma,\delta]f$ is defined on the additive subgroup
$\{h\in R\,|\, {\rm deg}(h)<m \}$ of $R$ by using right division by $f$
to define the algebra multiplication $g\circ h=gh \,\,{\rm mod}_r f $.
The properties of the algebras $S_f$ were studied in detail in
\cite{P66, P68}, and for $D$ a finite base field (hence w.l.o.g. $\delta=0$) in \cite{LS}.

Even earlier, the algebra $S_f$ with $f(t)=t^2-i\in \mathbb{C}[t;\can]$, $\can$ the complex conjugation, appeared in \cite{D06}
 as the first example of a nonassociative division algebra.

Although the algebras themselves have received little attention so far,
the right nucleus of $S_f$ (the \emph{eigenspace} of $f\in R$) already
appeared implicitly in classical constructions by  Amitsur  \cite{Am, Am2, Am3},
but also in results on computational aspects of operator algebras; they are for instance
 used in algorithms factoring skew polynomials over $\mathbb{F}_q(t)$ or finite fields, cf.
\cite{G0, GZ, GLN, G}. The role of classical algebraic constructions in coding theory is well known
(cf. \cite[Chapter~9]{N1}, \cite{N2, N3, N4}).

 Moreover, recently space-time block codes, coset codes and wire-tap codes were obtained employing the algebras $S_f$
 over number fields, cf. \cite{DO.0, DO, OS, Pu16.3, PS15.3, PS15.4, SPO12}, and they also appear useful for
 linear cyclic codes \cite{Pu16.1, Pu16.2}.

If $K$ is a finite field, $F$ the fixed field of $\sigma$, $K/F$ a finite Galois field extension
 and $f\in K[t;\sigma]=K[t;\sigma,0]$ irreducible and invariant, the $S_f$ are \emph{Jha-Johnson semifields}
  (also called \emph{cyclic semifields}) \cite[Theorem 15]{LS}, and were studied for instance by Wene \cite{W00} and more recently
 by Lavrauw and Sheekey \cite{LS}.  The main motivation for our paper  comes from the question how the automorphism groups of Jha-Johnson semifields look like.
  The  results presented here are applied to some Jha-Johnson semifields in \cite{BPS}.

The structure of this paper is as follows: In Section \ref{sec:prel},
we introduce the terminology and define the algebras $S_f$. We  limit our observations to the  algebras which are not
associative. Given a field extension $K$, $\sigma\in {\rm Aut}(K)$ of order $n$ with fixed field $F$, such that $\sigma$ commutes with
all  $\tau\in {\rm Aut}_F(K)$, and $f\in K[t;\sigma]$ of degree $m$ not invariant, we compute the automorphisms of $S_f$
 in Section \ref{sec:2}.  We obtain all automorphisms for  $n\geq m-1$ and some partial results for $n<m-1$
(Theorems \ref{thm:automorphism_of_Sf_field_case} and
\ref{thm:automorphism_of_Sf_field_caseII}). For $n\geq m-1$, the
automorphisms in ${\rm Aut}_F(S_f)$ are canonically
induced by the $F$-automorphism $G$ of $R=K[t;\sigma]$ which satisfy
$G(f(t))=af(t)$ for some $a\in K^\times$, and on $K$ restrict to an automorphism that commutes with $\tau$.

 The automorphisms groups of $S_f$ where
$f(t) = t^m - a \in K[t; \sigma]$, $a \in K \setminus F$, play a special role, as for
all nonassociative $S_g$ with $g(t) = t^m - \sum_{i=0}^{m-1} b_i t^i
\in K[t;\sigma]$ and $b_0= a$, ${\rm Aut}_F(S_g)$ is a
 subgroup of ${\rm Aut}_F(S_f)$ when $n\geq m-1$.

 We then focus on the situation that $K/F$ is a finite Galois field extension such that $\sigma$ commutes with
all  $\tau\in {\rm Gal}(K/F)$. In many cases,  either ${\rm Aut}_F(S_f)\cong {\rm Gal}(K/F)$ or is trivial
(Theorem \ref{thm:automorphism_of_Sf_field_caseIV}).
Necessary conditions for extending Galois automorphisms  $\tau\in{\rm Gal}(K/F)$  to $S_f$ are studied in Sections
 \ref{sec:nec} and \ref{subsec:I}.
The existence of cyclic subgroups of ${\rm Aut}_F(S_f)$
 is investigated in Section \ref{sec:cyclic_subgroups}.

For $f(t)=t^m-a\in K[t;\sigma]$ and $K/F$ a cyclic field extension of degree $m$,  the algebra $S_f$ is also called a
\emph{nonassociative cyclic algebra} and denoted by $(K/F,\sigma,a)$. These algebras are canonical
generalizations of associative cyclic algebras, but also generalizations of the algebras in \cite{Am2, J96}.
The automorphisms of
nonassociative cyclic algebras  are investigated in Section \ref{sec:nonasscyclic}.
 All the automorphisms of $A=(K/F,\sigma,a)$  extending $id_K$ are inner and form a cyclic subgroup of ${\rm Aut}_F(A)$
isomorphic to ${\rm ker}(N_{K/F})$. In some cases, this is the whole automorphism group, e.g.
if   $F$ has no $m$th root of unity. In these cases, every automorphism of $A$ leaves $K$ fixed and is inner.
We explain when the automorphism group of a nonassociative quaternion algebra $A$ (where $m=2$)
 contains a dicyclic group and when it contains a subgroup isomorphic to the semidirect product of two cyclic groups.

 In Section \ref{sec:iso} we briefly investigate  isomorphisms between two algebras $S_f$ and $S_g$.

 This work is part of the first author's PhD thesis \cite{CB} written under the supervision of the second author.
 For results on the automorphisms of the more general algebras defined using  $f\in D[t;\sigma]$,
 or a more detailed study and the (less relevant) cases left out in this paper
 the reader is referred to \cite{CB}. For examples of applications of the associated classical constructions
 the readers are referred to  \cite{N3, N4, N1, N2}.

%
%

\section{Preliminaries} \label{sec:prel}


\subsection{Nonassociative algebras} \label{subsec:nonassalgs}


Let $F$ be a field and let $A$ be an $F$-vector space. $A$ is an
\emph{algebra} over $F$ if there exists an $F$-bilinear map $A\times
A\to A$, $(x,y) \mapsto x \cdot y$, denoted simply by juxtaposition
$xy$, the  \emph{multiplication} of $A$. An algebra $A$ is called
\emph{unital} if there is an element in $A$, denoted by 1, such that
$1x=x1=x$ for all $x\in A$. We will only consider unital algebras
from now on without explicitly saying so.

Associativity in $A$ is measured by the {\it associator} $[x, y, z] =
(xy) z - x (yz)$. The {\it left nucleus} of $A$ is defined as ${\rm
Nuc}_l(A) = \{ x \in A \, \vert \, [x, A, A]  = 0 \}$, the {\it
middle nucleus} of $A$ is ${\rm Nuc}_m(A) = \{ x \in A \, \vert \,
[A, x, A]  = 0 \}$ and  the {\it right nucleus} of $A$ is defined as
${\rm Nuc}_r(A) = \{ x \in A \, \vert \, [A,A, x]  = 0 \}$. ${\rm
Nuc}_l(A)$, ${\rm Nuc}_m(A)$, and ${\rm Nuc}_r(A)$ are associative
subalgebras of $A$. Their intersection
 ${\rm Nuc}(A) = \{ x \in A \, \vert \, [x, A, A] = [A, x, A] = [A,A, x] = 0 \}$ is the {\it nucleus} of $A$.
${\rm Nuc}(A)$ is an associative subalgebra of $A$ containing $F1$
and $x(yz) = (xy) z$ whenever one of the elements $x, y, z$ lies in
${\rm Nuc}(A)$.  The
 {\it center} of $A$ is ${\rm C}(A)=\{x\in A\,|\, x\in \text{Nuc}(A) \text{ and }xy=yx \text{ for all }y\in A\}$.

An $F$-algebra $A\not=0$ is called a \emph{division algebra} if for any
$a\in A$, $a\not=0$, the left multiplication  with $a$, $L_a(x)=ax$,
and the right multiplication with $a$, $R_a(x)=xa$, are bijective.
If $A$ has finite dimension over $F$, then $A$ is a division algebra if
and only if $A$ has no zero divisors \cite[ pp. 15, 16]{Sch}.
An element $0\not=a\in A$ has a \emph{left inverse} $a_l\in A$, if
$R_a(a_l)=a_l a=1$, and a \emph{right inverse}
 $a_r\in A$, if $L_a(a_r)=a a_r=1$.

An automorphism $G\in {\rm Aut}_F(A)$ is an \emph{inner automorphism}
if there is an element $m\in A$ with left inverse $m_l$ such
that $G(x) = (m_lx)m$ for all $x\in A$. Given an inner automorphism
$G_m\in {\rm Aut}_F(A)$ and some $H \in {\rm Aut}_F(A)$, then clearly
$H^{-1}\circ G_m \circ H\in {\rm Aut}_F(A)$ is an inner automorphism.
\cite[Lemma 2, Theorem 3, 4]{W09} generalize to any
nonassociative algebra:

\begin{proposition} \label{prop:inner_Wene}
Let $A$ be an algebra over $F$.
\\ (i) For all invertible $n\in {\rm Nuc}(A)$, $G_n(x)=(n^{-1}x)n$ is an inner automorphism of $A$.
\\ (ii) If $G_m$ is an inner automorphism of $A$, then so is $G_{nm}(x) = ((m_l n^{-1})x)(nm)$ for all
invertible  $n\in {\rm Nuc}(A)$.
\\ (iii) If $G_m$ is an inner automorphism of $A$, and $a,b\in {\rm Nuc}(A)$ are invertible, then
$G_{am}=G_{bm}$ if and only if $ ab^{-1}\in C(A).$
\\ (iv) For invertible $n,m\in {\rm Nuc}(A)$, $G_m=G_n$ if and only if $n^{-1}m\in C(A)$.
\end{proposition}

The set $\{G_m\,|\, m\in {\rm Nuc}(A) \text{ invertible} \}$ is a
subgroup of ${\rm Aut}_F(A)$. For each invertible $m\in {\rm Nuc}(A)\setminus C(A)$, $G_m$ generates a cyclic subgroup
which has finite order $s$ if $m^s\in C(A)$, so in particular if $m$ has order $s$.

Note that if the nucleus is commutative, then for all invertible $n\in
{\rm Nuc}(A)$, $G_n(x)=(n^{-1}x)n$ is an inner automorphism of $A$
such that ${G_n}|_{{\rm Nuc}(A)}=id_{{\rm Nuc}(A)}$.


\subsection{Twisted polynomial rings}


Let $K$ be a field and $\sigma$ an automorphism of $K$.
 The \emph{twisted polynomial ring} $K[t;\sigma]$
is the set of polynomials $a_0+a_1t+\dots +a_nt^n$ with $a_i\in K$,
where addition is defined term-wise and multiplication by
$ta=\sigma(a)t $
for all $a\in K$.
For $f=a_0+a_1t+\dots +a_nt^n$ with $a_n\not=0$ define ${\rm
deg}(f)=n$ and put ${\rm deg}(0)=-\infty$. Then ${\rm deg}(fg)={\rm deg}
(f)+{\rm deg}(g).$
 An element $f\in R$ is \emph{irreducible} in $R$ if it is not a unit and  it has no proper factors,
 i.e. if there do not exist $g,h\in R$ with
 ${\rm deg}(g),{\rm deg} (h)<{\rm deg}(f)$ such
 that $f=gh$.

 $R=K[t;\sigma]$ is a left and right principal ideal domain
 and there is a right division algorithm in $R$: for all
$g,f\in R$, $g\not=0$, there exist unique $r,q\in R$ with ${\rm
deg}(r)<{\rm deg}(f)$, such that $g=qf+r.$ There is also a left
division algorithm in $R$ \cite[p.~3 and Prop. 1.1.14]{J96}. (Our
terminology is the one used by Petit \cite{P66} and Lavrauw and
Sheekey \cite{LS}; it is different from Jacobson's, who calls what we
call right a left division algorithm and vice versa.)
Define $F={\rm Fix}(\sigma)$.


\subsection{Nonassociative algebras obtained from skew polynomial rings} \label{subsec:structure}


Let $K$ be a field, $\sigma$ an automorphism of $K$ with $F={\rm Fix}(\sigma)$, and $f\in R=K[t;\sigma]$ of degree $m$. Let ${\rm mod}_r f$ denote the remainder of right division by $f$.
Then then additive abelian group
 $R_m=\{g\in K[t;\sigma]\,|\, {\rm deg}(g)<m\}$
 together with the multiplication
 $g\circ h=gh \,\,{\rm mod}_r f $
 is a unital nonassociative algebra $S_f=(R_m,\circ)$ over
 $F_0=\{a\in K\,|\, ah=ha \text{ for all } h\in S_f\}.$
 $F_0$  is a subfield of $K$ \cite[(7)]{P66} and it is straightforward to see that if $f(t)=t^m-\sum_{i=0}^{m-1}a_it^i$
  and $a_0\not=0$ then $F_0=F$ \cite[Remark 9]{Pu16.1}.
 The algebra $S_f$ is also denoted by $R/Rf$ \cite{P66,P68}
 if we want to make clear which ring $R$ is involved in the construction.
 In the following, we call the algebras $S_f$  \emph{Petit algebras} and denote their multiplication simply by juxtaposition.

  Using left division by $f$ and
  the remainder ${\rm mod}_l f$ of left division by $f$ instead, we can analogously define the multiplication for
  another unital nonassociative algebra on $R_m$ over $F_0$, called $\,_fS$. We will only consider the Petit
  algebras $S_f$,  since every algebra
$\,_fS$ is the opposite algebra of some Petit algebra \cite[(1)]{P66}.

\begin{theorem} (cf. \cite[(2),  (5), (9)]{P66})  \label{Properties of S_f petit}
Let $f(t) \in R = K[t;\sigma]$.
\\
(i) If $S_f$ is not associative then ${\rm Nuc}_l(S_f)={\rm
Nuc}_m(S_f)=K$ and ${\rm Nuc}_r(S_f)=\{g\in R\,|\, fg\in Rf\}.$
\\ (ii) The powers of $t$ are associative if and only if $t^mt=tt^m$
 if and only if $t\in {\rm Nuc}_r(S_f)$ if and only if $ft\in Rf.$
\\ (iii) Let $f\in R$ be irreducible and $S_f$ a finite-dimensional $F$-vector space
 or free of finite rank as a right ${\rm Nuc}_r(S_f)$-module. Then $S_f$
is a division algebra.
\\ Conversely, if $S_f$ is a division algebra then $f$ is irreducible.
\\
(iv) $S_f$ is associative if and only if $f$ is invariant.
In that case, $S_f$ is the usual quotient algebra.
\\ (v)
 Let $f(t)=t^m-\sum_{i=0}^{m-1}a_it^i\in R=K[t;\sigma]$.
Then $f$ is invariant  if and only if
$\sigma^m(z)a_i=a_i\sigma^i(z)$ for all $z\in K$, $i\in\{0,\dots,m-1\}$ and $a_i\in F$
for all $i\in\{0,\dots,m-1\}$.
\end{theorem}

Note that if $f$ is not invariant, then
the nucleus of any $S_f=K[t;\sigma]/K[t;\sigma]f$ is a subfield of $K={\rm Nuc}_l(S_f)$.
If ${\rm Nuc}(S_f)$ is larger than  $F$, then
$\{G_m\,|\, 0\not=m\in {\rm Nuc}(A) \}$ is a non-trivial
subgroup of ${\rm Aut}_F(S_f)$ and each inner automorphism $G_m$ in this subgroup
extends $id_{{\rm Nuc}(A)}$ by Proposition \ref{prop:inner_Wene}.

 \begin{proposition}\label{prop:subfields}
 Let $f(t) \in F[t]=F[t;\sigma]\subset K[t;\sigma]$.
 \\ (i) $F[t]/(f(t))$ is a commutative subring of $S_f$ and
$F[t]/(f(t))\cong F\oplus Ft\oplus\dots\oplus Ft^{m-1}\subset {\rm Nuc}_r(S_f).$
In particular, then $ft\in Rf$ which is equivalent to the powers of $t$ being associative,
which again is equivalent to $t^mt=tt^m$.
\\ (ii) If $f(t)$ is irreducible in $F[t]$, $F[t]/(f(t))$ is an algebraic subfield of degree $m$ contained
in the right nucleus.
\end{proposition}

\begin{proof}
$S_f$ contains the commutative subring $F[t]/(f(t)).$ If $f(t)$ is irreducible in $F[t]$, this is an algebraic
field extension
of $F$. This subring is isomorphic to the ring consisting of the elements $\sum_{i=0}^{m-1}a_it^i$
with $a_i\in F$.

Clearly $F\subset {\rm Nuc}_r(S_f)$. For all $a,b,c\in K$, $i,j\in\{0,\dots, m-1\}$
 we have
$[at^i,bt^j,t]=(a\sigma^i(b)t^{i+j})t-(at^i)(bt^{j+1})=a\sigma^i(b)t^{i+j+1}-a\sigma^i(bc) t^{i+j}=0.$
Thus $t\in {\rm Nuc}_r(S_f)$ which implies that $F\oplus Ft\oplus\dots\oplus Ft^{m-1}\subset  {\rm Nuc}_r(S_f)$,
hence the assertion. The rest is obvious.
\end{proof}

We will assume throughout the paper that $\text{deg}(f) = m \geq
2$ (since if $f$ is constant then $S_f\cong K$) and that
$\sigma\not=id$. Without loss of generality, we only consider monic polynomials $f$, since $S_f= S_{af}$ for all
non-zero $a\in K$.

%
%

\section{Automorphisms of $S_f$}\label{Automorphisms of S_f} \label{sec:2}

\subsection{} Let $K$ be a field, $\sigma$ an automorphism of $K$ of order $n$ (which may be infinite),
 $F={\rm Fix}(\sigma)$, and
$f(t) = t^m - \sum_{i=0}^{m-1} a_i t^i \in K[t;\sigma]$
a twisted polynomial which is  not invariant.

 \begin{theorem} \label{thm:automorphism_of_Sf_field_case}
Suppose $\sigma$ commutes with all $\tau\in {\rm Aut}_F(K)$. Let  $n
\geq m-1$. Then  $H$ is an
automorphism of $S_f$ if and only if $H=H_{\tau , k}$ with
$$H_{\tau , k}(\sum_{i=0}^{m-1} x_i t^i )= \tau(x_0) + \tau(x_1)kt + \tau(x_2)k\sigma(k)t^2+\cdots +
\tau(x_{m-1}) k \sigma(k) \cdots \sigma^{m-2}(k) t^{m-1},$$
where  $\tau\in {\rm Aut}_F(K)$ and $k \in K^{\times}$ is such that
\begin{equation} \label{eqn:neccessary}
\tau(a_i) = \Big( \prod_{l=i}^{m-1}\sigma^l(k) \Big) a_i
\end{equation}
for all $i \in\{ 0, \ldots, m-1\}$.
\end{theorem}

\begin{proof}
Let $H: S_f \rightarrow S_f$ be an automorphism. Since $S_f$ is not
associative, $\text{Nuc}_l(S_f) = K$ by Theorem  \ref{Properties of
S_f petit} (i). Since any automorphism  preserves the left nucleus,  $H(K) = K$ and so $H \vert_K = \tau$ for some
 $\tau \in \text{Aut}_F(K)$. Suppose $H(t) = \sum_{i=0}^{m-1} k_i t^i$ for some $k_i \in K$.
Then we have
\begin{equation} \label{eqn:automorphism_necessity_theorem 1}
H(tz) = H(t)H(z) = (\sum_{i=0}^{m-1} k_i t^i ) \tau(z) =
\sum_{i=0}^{m-1} k_i \sigma^{i}(\tau(z)) t^i \end{equation} and
\begin{equation} \label{eqn:automorphism_necessity_theorem 2} H(tz) =
H(\sigma(z)t) = \sum_{i=0}^{m-1} \tau(\sigma(z)) k_i t^i
\end{equation} for all $z \in K$. Comparing the coefficients of $t^i$
in \eqref{eqn:automorphism_necessity_theorem 1} and
\eqref{eqn:automorphism_necessity_theorem 2} we obtain
\begin{equation} \label{eqn:automorphism_necessity_theorem 3} k_i \sigma^{i}(\tau(z)) =
k_i \tau(\sigma^i(z))= \tau(\sigma(z)) k_i=k_i\tau(\sigma(z))
\end{equation} for all $i \in\{ 0, \ldots, m-1\}$ and all $z \in K$. This implies
$k_i (\tau\big( \sigma^i(z) - \sigma(z) \big) )=0$
for all $i \in\{ 0, \ldots, m-1 \}$ and all $z \in K$ since $\sigma$
and $\tau$ commute, i.e.
\begin{equation} \label{eqn:automorphism_necessity_theorem 4}
k_i=0 \text{ or } \sigma^{i}(z)=\sigma(z) \end{equation} for all $i
\in\{ 0, \ldots, m-1\}$ and all $z \in K$.

Since $\sigma$ has order $n\geq m-1$, which means $\sigma^i\not=\sigma$ for
all $i\in\{ 0, \ldots, m-1\}$, $i\not=1$,
\eqref{eqn:automorphism_necessity_theorem 4} implies $k_i = 0$ for
all $ i \in\{ 0, \ldots, m-1\}$, $i\not=1$. Therefore $H(t) = kt$ for some
$k \in K^{\times}$. Furthermore, we have
$
H(z t^i) = H(z)H(t)^i = \tau(z) (kt)^i = \tau(z) \Big(\prod_{l=0}^{i-1} \sigma^l(k) \Big) t^i$
 for all $i
\in \{ 1, \ldots, m-1 \}$ and $z \in K$. Thus $H$ has the form
\begin{equation} \label{eqn:automain}
H_{\tau , k}(\sum_{i=0}^{m-1} x_i t^i )=  \tau(x_0) + \sum_{i=1}^{m-1} \tau(x_i)
\prod_{l=0}^{i-1}\sigma^l(k) t^i,
\end{equation}
for some $k \in K^{\times}$. Moreover, with $t^m=t t^{m-1}$, also
\begin{equation} \label{eqn:automorphism_necessity_theorem 8}
H(t^m) = H \Big( \sum_{i=0}^{m-1} a_i t^i \Big) = \sum_{i=0}^{m-1}
H(a_i) H(t)^i
= \tau(a_0) + \sum_{i=1}^{m-1} \tau(a_i) \big( \prod_{l=0}^{i-1} \sigma^l(k) \big) t^i
\end{equation}
and  $H(t t^{m-1})=H(t)H(t^{m-1})=H(t)H(t)^{m-1}$, i.e.
\begin{equation} \label{eqn:automorphism_necessity_theorem 9}
H(t)^m =H(t)H(t)^{m-1}= k \sigma(k) \cdots \sigma^{m-1}(k) t^{m} =
k \sigma(k) \cdots \sigma^{m-1}(k)\sum_{i=0}^{m-1} a_i t^i.
\end{equation}
Comparing \eqref{eqn:automorphism_necessity_theorem 8} and
\eqref{eqn:automorphism_necessity_theorem 9} gives $  \tau(a_i) =
\Big( \prod_{q=i}^{m-1} \sigma^q(k) \Big) a_i$ for all $i \in\{ 0,
\ldots, m-1\}$. Thus $H$ is as in (\ref{eqn:automain})
where $k \in K^{\times}$ is such that \eqref{eqn:neccessary} holds
for all $i \in\{ 0, \ldots, m-1 \}$.

The $H_{\tau,k}$ are indeed automorphisms of $S_f$:
 Let $G$ be an automorphism of $R=K[t;\sigma]$. Then for $h(t) = \sum_{i=0}^{r} b_i t^i \in K[t; \sigma]$ we have
$G(h(t))=
\tau(b_0)  +\sum_{i=i}^{m-1}\tau(b_i) \prod_{l=0}^{i-1}\sigma^l(k) t^i$
for some $\tau\in{\rm Aut}(K)$  such that
$\sigma\circ\tau=\tau\circ\sigma$ and some $k\in K^\times$
(the proof of  \cite[Lemma 1]{LS} works for any
$R=K[t;\sigma]$, or cf. \cite[p.~75]{K}). It is straightforward to see that $S_f\cong S_{G(f)}$ (cf. \cite[Theorem 7]{LS},
the proof works for any $R=K[t;\sigma]$).
In particular, this means that if $k\in K^\times$ satisfies  \eqref{eqn:neccessary} then
$G(f(t)) = \big( \prod_{l=0}^{m-1} \sigma^l(k) \big) f(t)=af(t)$
with $a\in K^\times$ being the product of the $\sigma^l(k)$, and so $G$ induces an isomorphism of $S_f$
with $ S_{af}=S_f$, i.e. an automorphism of $S_f$.
 The automorphisms  of ${\rm Aut}_F(S_f)$ are therefore all canonically induced by the $F$-automorphisms $G$
of $R=K[t;\sigma]$ which satisfy \eqref{eqn:neccessary}.
\end{proof}

The assumption that $n \geq m-1$ is needed in  (\ref{eqn:automorphism_necessity_theorem 3}) to conclude that $k_i =
0$ for $i = 0, 2, 3, \ldots , m-1$ and so $H(t)=kt$. If $n < m-1$ we still obtain:

\begin{theorem} \label{thm:automorphism_of_Sf_field_caseII}
Suppose $\sigma$ commutes with all  $\tau\in {\rm Aut}_F(K)$. Let  $n< m-1$.
\\ (i) For all  $k \in K^{\times}$ satisfying  (\ref{eqn:neccessary}) for all $i \in\{ 0, \ldots, m-1\}$,
 the maps  $H_{\tau , k}$ from Theorem
\ref{thm:automorphism_of_Sf_field_case} are automorphisms of $S_f$ and form a subgroup of  ${\rm Aut}_F(S_f)$.
\\ (ii)
Let  $H\in {\rm Aut}_F(S_f)$  and $N={\rm Nuc}_r(S_f)$. Then $H
\vert_K = \tau$ for some $\tau\in {\rm Aut}_F(K)$, $H \vert_{N} \in
{\rm Aut}_F(N)$ and $H(t)=g(t)$ with $g(t) = k_1t + k_{1+n}t^{1+n} +
k_{1+2n}t^{1+2n} + \ldots + k_{1+sn} t^{1+sn}$
 for some $k_{1+ln}\in K$, $0\leq l\leq s$.
 Moreover, $g(t)^i$ is well defined for all $i\leq m-1$, i.e., all powers of $g(t)$ are associative for all
$i\leq m-1$, and
 $g(t)g(t)^{m-1} = \sum_{i=0}^{m-1}\tau(a_i)g(t)^i.$
 Thus
 $$H(\sum_{i=0}^{m-1} x_i t^i )=\sum_{i=0}^{m-1} \tau(x_i) g(t)^i.$$
\end{theorem}

\begin{proof}
(i) is straightforward, using the relevant parts of the proof of Theorem
\ref{thm:automorphism_of_Sf_field_case}. Note that the inverse of
$H_{\tau,k}$ is $H_{\tau^{-1},\tau^{-1}(k^{-1})}$
and $H_{\tau,k} \circ H_{\rho,b} =H_{\tau \rho, \tau(b)k}.$
\\ (ii)
Let  $H: S_f \rightarrow S_f$ be an automorphism. As in Theorem
\ref{thm:automorphism_of_Sf_field_case}, $H \vert_K = \tau$ for some
$\tau\in {\rm Aut}_F(K)$, and $H \vert_{N} \in {\rm Aut}_F(N)$. Suppose $H(t) = \sum_{i=0}^{m-1} k_i t^i$ for some
$k_i\in K$.  Comparing the coefficients of $t$ in $H(tz) = H(t)H(z)=
H(\sigma(z)t)$ we obtain (\ref{eqn:automorphism_necessity_theorem 4})
for all $i \in\{ 0, \ldots, m-1\}$ and all $z \in K$.
 Since $\sigma$ has order $n<m-1$,
here, (\ref{eqn:automorphism_necessity_theorem 4})  only implies $k_i
= 0$ for $i \in\{ 0, \ldots, n\}$, $i\not=1$. Therefore $H(t) =
k_1t+\sum_{i=n+1}^{m-1} k_i t^i$
 for some  $ k_i\in K$.
However, $\sigma^i(z) = \sigma(z)$ for all $z \in K$ if and only if
 $i = nl+1$ for some $l \in \mathbb{Z}$ since $\sigma$ has order $n$.
 Therefore \eqref{eqn:automorphism_necessity_theorem 4} implies $k_i = 0$
 for every $i \neq 1 + nl$, $l \in \mathbb{N}_0$, $i \in  \left\{ 0, \ldots, m-1\right\}$. Thus
$H(t) = k_1t + k_{1+n}t^{n+1} + \ldots + k_{1+sn} t^{1+sn}$ for
some $s$, $sn<m-1$. Furthermore, $H(t^m) = H(\sum_{i=0}^{m-1} a_i
t^i) = \sum_{i=0}^{m-1} \tau(a_i)  (k_1t + k_{1+n}t^{1+n} +
\ldots +
 k_{1+sn} t^{1+sn})^i$
 and
 $H(t^m) = (k_1t + k_{1+n}t^{1+n} + \ldots + k_{1+sn} t^{1+sn})^m.$
Similarly, $H(t)^i = (k_1t + k_{1+n}t^{1+n} + \ldots + k_{1+sn}
t^{1+sn})^i.$ Together these imply the assertion.
\end{proof}

A closer look at the proof of Theorems \ref{thm:automorphism_of_Sf_field_case} and \ref{thm:automorphism_of_Sf_field_caseII}
reveals that in fact the following holds
without requiring $\sigma$ to commute with all $\tau\in {\rm Aut}_F(K)$:

\begin{proposition} \label{prop:automorphism_of_Sf_field_general}
 (i) For every $k \in K^{\times}$  satisfying  (\ref{eqn:neccessary})  for all $i \in\{ 0, \ldots, m-1\}$
 for $\tau=id$, $H_{id,k}$
 is an automorphism of $S_f$ and generates a subgroup of ${\rm Aut}_F(S_f)$.
 \\ (ii)
 If any $H\in {\rm Aut}_F(S_f)$ restricts to some $\tau\in {\rm Aut}_F(K)$ such that $\tau\circ\sigma=
 \sigma\circ\tau$ then $H=H_{\tau , k}$ with $k\in K^\times$ as in Theorem \ref{thm:automorphism_of_Sf_field_case}. Moreover,
 $ \{H_{\tau , k}\,|\, \tau\in{\rm Aut}_F(S_f), \tau\circ\sigma= \sigma\circ\tau, k\in K^\times \text{
with } \tau(a_i) = ( \prod_{l=i}^{m-1}\sigma^l(k) ) a_i \text{ for
all } i \in\{ 0, \ldots, m-1\}\}$ is a subgroup of ${\rm
Aut}_F(S_f)$.
\\ (iii) If $m=2$, $H\in {\rm Aut}(S_f)$ if and only if $H=H_{\tau,k}$ with $\tau\circ\sigma=\sigma\circ\tau$,
$\tau(a_0)=k\sigma(k)a_0,$ and $ \tau(a_1)=\sigma(k)a_1.$
\end{proposition}

\subsection{}
 The  automorphisms groups of $S_f$ for $f(t) = t^m - a \in K[t; \sigma]$, $a \in K \setminus F$,
 are crucial in the understanding of the automorphism groups of all the algebras $S_g$,
 as for all nonassociative $S_g$ with $g(t) = t^m - \sum_{i=0}^{m-1} b_i t^i\in K[t;\sigma]$ such that $b_0= a$, ${\rm Aut}_F(S_g)$ is a
 subgroup of ${\rm Aut}_F(S_f)$:

\begin{theorem} \label{Aut(S_f) subgroup corollary}
Suppose $\sigma$ commutes with all $\tau\in {\rm Gal}(K/F)$. Let $n\geq m-1$ and
$g(t) = t^m - \sum_{i=0}^{m-1} b_i t^i \in K[t;\sigma]$ not be invariant.
\\ (i) If $f(t) = t^m - b_0 \in K[t; \sigma]$, $b_0 \in K \setminus F$, then
${\rm Aut}_F(S_g)\subset {\rm Aut}_F(S_f)$
 is a subgroup.
\\ (ii) Let $f(t) = t^m - \sum_{i=0}^{m-1} a_i t^i \in K[t;\sigma]$ not be invariant
and assume $b_i \in \left\{ 0 , a_i \right\}$ for all $i \in \{ 0,\ldots , m-1\}$.
 Then
${\rm Aut}_F(S_g)\subset {\rm Aut}_F(S_f)$
 is a subgroup.
\end{theorem}

\begin{proof}
(i) Let $H \in {\rm Aut}_F(S_g)$. By Theorem
\ref{thm:automorphism_of_Sf_field_case}, $H$ has the form
$H(\sum_{i=0}^{m-1} x_i t^i )=  \tau(x_0) + \sum_{i=1}^{m-1} \tau(x_i)
\prod_{l=0}^{i-1}\sigma^l(k) t^i,$
where $\tau \in {\rm Aut}_F(K)$ and $k \in K^{\times}$ satisfy $\tau(b_i) = \Big(
\prod_{j=i}^{m-1} \sigma^j(k) \Big) b_i$ for all $i = 0, \ldots,
m-1$. In particular, $\tau(b_0) = k \sigma(k) \cdots
\sigma^{m-1}(k)b_0$ and so $H$ is also an automorphism of $S_f$,
again  by Theorem \ref{thm:automorphism_of_Sf_field_case}.
\\ (ii) The proof is analogous to (i).
\end{proof}

Similarily, for $n < m-1$ employing Theorem \ref{thm:automorphism_of_Sf_field_caseII} we obtain:

\begin{theorem} \label{Aut(S_f) subgroup corollaryII}
Suppose $\sigma$ commutes with all $\tau\in {\rm Aut}_F(K)$. Let $n < m-1$ and
$g(t) = t^m - \sum_{i=0}^{m-1} b_i t^i \in K[t; \sigma]$ not
be invariant.
\\ (i) If $f(t) = t^m - b_0 \in K[t; \sigma]$, $b_0 \in K \setminus F$, then
$\{H\in {\rm Aut}_F(S_g)\,|\, H=H_{\tau , k} \}$ is a subgroup
of $\{H\in {\rm Aut}_F(S_f)\,|\, H=H_{\tau , k} \}.$
\\ (ii) Let $f(t) = t^m - \sum_{i=0}^{m-1} a_i t^i \in K[t;\sigma]$ not be invariant
such that $b_i \in \left\{ 0 , a_i \right\}$ for all $i \in \{ 0,\ldots , m-1\}$
 Then  $\{H\in{\rm Aut}_F(S_g)\,|\, H=H_{\tau , k}
\}$ is a subgroup of $\{H\in {\rm Aut}_F(S_f)\,|\, H=H_{\tau ,
k} \}.$
\end{theorem}

The automorphism groups of $S_f$ with $f(t)=t^m-a\in K[t;\sigma]$   are therefore particularly relevant.

%
%

\section{Necessary conditions for extending Galois automorphisms to $S_f$}\label{sec:nec}

From now on we restrict ourselves to the situation that  $R=K[t;\sigma]$ and $F={\rm Fix}(\sigma)$,
 where $K/F$ is a finite Galois field extension and  $\sigma$ of order $n$.

We  take a closer look at Equality
(\ref{eqn:neccessary}), which gives necessary conditions on how to
choose the elements $k\in K^\times$ used to extend $\tau\in {\rm Gal}(K/F)$ to ${\rm
Aut}_F(S_f)$. These become more restrictive for the choice of the elements $k$,
the more coefficients in $f(t)$ are non-zero.
 Let $N_{K/F} : K \rightarrow F$ be the norm of $K/F$.
 All monic polynomials $f$ considered in the following are assumed to not be invariant and of degree $m$.

\begin{proposition}\label{prop:conditionsonk}
Suppose that $\sigma$ and $\tau$ commute. Let  $f(t) = t^m - \sum_{i=0}^{m-1} a_i t^i \in K[t; \sigma]$ and $k
\in K^{\times}$  such that $$(1) \quad \quad\quad \tau(a_i) = \Big(
\prod_{l=i}^{m-1}\sigma^l(k) \Big) a_i$$ for all $i \in\{ 0, \ldots,
m-1\}$. Then:
\\ (i) For all $i \in\{ 0, \ldots, m-1\}$ with $a_i\not=0$, $N_{K/F}(k)$ is an $(m-i)$th root of unity.
\\
 In particular, if  $a_0\not=0$ (e.g., if $f(t)$ is irreducible) then $N_{K/F}(k)$ is an $m$th root of unity,
 and if $a_{m-1}\not=0$ then $N_{K/F}(k)  =1$.
\\
 If  $a_{m-1}\in  {\rm Fix}(\tau)^\times$ then $k=1$.
 \\ (ii)
 If $\tau \neq id$ and there is some $i$ such that $a_i$ is not contained in
  ${\rm Fix}(\tau)$, then $k\not=1$.
 \\ (iii)  Suppose that there is some  $a_i\not=0$ and $F$ does not contain any
  non-trivial $(m-i)$th roots of unity.
Then $N_{K/F}(k)=1$.
\\ (iv) If there is an $i \in\{ 0, \ldots, m-1\}$ such that $a_i\in {\rm Fix}(\tau)^\times$, then
$1 = \prod_{l=i}^{m-1}\sigma^l(k).$
 In particular, if $n=m$,  $\sigma$ generates ${\rm Gal}(K/F)$, and $a_0\in {\rm Fix}(\tau)^\times$ then $k\in{\rm ker}(N_{K/F})$.
\\ (v) Suppose $\tau=id_K$. Then for all $i \in\{ 0, \ldots, m-1\}$ with $a_i\not=0$,
$1 =  \prod_{l=i}^{m-1}\sigma^l(k).$
 In particular, if $n=m$, $\sigma$ generates ${\rm Gal}(K/F)$ and $a_0\not=0$ then $k\in{\rm ker}(N_{K/F})$. In this case,
 the automorphisms extending $id_K$  are in one-one correspondence with those
 $k\in {\rm ker}(N_{K/F})$ satisfying  (\ref{eqn:neccessary}).
\end{proposition}

\begin{proof}
(i) Equality (\ref{eqn:neccessary}) states that $\tau(a_i) = \Big(
\prod_{l=i}^{m-1}\sigma^l(k) \Big) a_i$ for all $i \in\{ 0, \ldots,
m-1\}$. Thus $N_{K/F}(a_i) =  \prod_{l=i}^{m-1}N_{K/F}(\sigma^l(k))
N_{K/F}(a_i)$
 (apply $N_{K/F}$ to both sides of (\ref{eqn:neccessary})), and therefore
$N_{K/F}(a_i) =  N_{K/F}(k)^{m-i} N_{K/F}(a_i)$ for all $i \in\{ 0,
\ldots, m-1\}$ is a necessary condition on $k$.  For all $a_i\not=0$,
this yields $1 =  N_{K/F}(k)^{m-i}$ therefore $N_{K/F}(k)\in
F^\times$ must be an $(m-i)$th root of unity, for all $i \in\{ 0,
\ldots, m-1\}$, with $a_i\not=0$.
Hence if $a_{m-1}\not=0$ then $\tau(a_{m-1}) = \sigma^{m-1}(k)  a_{m-1}$, thus
$N_{K/F}(a_{m-1}) = N_{K/F}(k)  N_{K/F}(a_{m-1})$, i.e. $N_{K/F}(k)  =1$.
 If even $a_{m-1}\in  {\rm Fix}(\tau)^\times$ then $a_{m-1} = \sigma^{m-1}(k)  a_{m-1}$ means $\sigma^{m-1}(k)=1$, i.e. $k=1$.
\\ (ii)  $k=1$ implies $\tau(a_i) = a_i$, i.e. $a_i\in {\rm Fix}(\tau)$ for all $i \in\{ 0, \ldots,
m-1\}$.
\\ (iii) By (i), $N_{K/F}(k)\in F^\times$ is an $(m-i)$th root of unity, for all $i \in\{ 0, \ldots, m-1\}$
 with $a_i\not=0$.
If $F$ does not contain any non-trivial $(m-i)$th roots of unity, then
$N_{K/F}(k)=1$.
\\ (iv) If there is an $i \in\{ 0, \ldots, m-1\}$ such that $a_i\in {\rm Fix}(\tau)^\times$, then  (\ref{eqn:neccessary})
 becomes
$1 = \prod_{l=i}^{m-1}\sigma^l(k) $. In particular, if  $a_0\in
{\rm Fix}(\tau)^\times$, $m=n$ and $\sigma$ generates ${\rm Gal}(K/F)$, then
$N_{K/F}(k)=1$ is a necessary condition on $k$.
\\ (v) Here,  (\ref{eqn:neccessary}) becomes
 $1 =  \prod_{l=i}^{m-1}\sigma^l(k)$  for all $i \in\{ 0, \ldots, m-1\}$
 with $a_i\not=0$. In particular,
if $n=m$,  $\sigma$ generates ${\rm Gal}(K/F)$ and $a_0\not=0$ (which happens if $f(t)$ is irreducible)
then
 $N_{K/F}(k)=1$
 is a necessary condition on $k$.
\end{proof}

For instance, Proposition \ref{prop:conditionsonk} (i) yields that $k=1$ if  $a_{m-1}\in  {\rm Fix}(\tau)^\times$
 and so Theorems \ref{thm:automorphism_of_Sf_field_case} and \ref{thm:automorphism_of_Sf_field_caseII} become:

 \begin{theorem} \label{thm:automorphism_of_Sf_field_caseIV}
Suppose $\sigma$ commutes with all $\tau\in {\rm Gal}(K/F)$ and $f(t) = t^m - \sum_{i=0}^{m-1} a_i t^i \in K[t;\sigma]$
is  not invariant  with $a_{m-1}\in F^\times$.
\\ (i) Let  $n\geq m-1$.  If $a_i\not\in {\rm Fix}(\tau)$ for all $\tau\not=id$ and all
non-zero $a_i$, $i\not=m-1$, then ${\rm Aut}_F(S_f)=\{id\}.$
\\
If $f(t)\in F[t;\sigma]$, any automorphism $H$ of $S_f$ has the form $H_{\tau , 1}$
where  $\tau\in {\rm Gal}(K/F)$,
and
${\rm Aut}_F(S_f)\cong {\rm Gal}(K/F).$
\\ (ii) Let  $n < m-1$.  If $f(t)\in F[t;\sigma]$ is not invariant,
 the maps  $H_{\tau , 1}$  are automorphisms of $S_f$  for all
$\tau\in {\rm Gal}(K/F) $ and
 ${\rm Gal}(K/F)$ is isomorphic to a subgroup of  ${\rm Aut}_F(S_f)$.
\end{theorem}

\begin{proof} (i) $H$ is an
automorphism of $S_f$ if and only if $H$ has the form $H_{\tau , k}$,
where  $\tau\in {\rm Gal}(K/F)$ and $k \in K^{\times}$ is such that
$\tau(a_i) = \Big( \prod_{l=i}^{m-1}\sigma^l(k) \Big) a_i$
for all $i \in\{ 0, \ldots, m-1\}$. Since $a_{m-1}\in F^\times$ we have $ a_{m-1}\in {\rm Fix}(\tau)^\times$ for all $\tau$ which
forces $k=1$ as the only possibility for any
 $\tau\in {\rm Gal}(K/F)$ by Proposition \ref{prop:conditionsonk} (i). This in turn means that any extension
$H_{\tau , k}$ has the form $H_{\tau , 1}$. In particular, the existence of an extension $H_{\tau , k}$, $\tau\not=id$,
  implies
$\tau(a_i) =  a_i$ for all non-zero $a_i$, $i\not=m-1$, that is $a_i\in{\rm Fix}(\tau)$
for all non-zero $a_i$.

Thus if  $a_i\not\in {\rm Fix}(\tau)$ for all $\tau\not=id$ and all $i \in\{ 0, \ldots, m-2\}$   then
there is no non-trivial $\tau$ that extends to an automorphism
of $S_f$ and ${\rm Aut}_F(S_f)= \{H_{id , 1}\}=\{id\}$.

If  $f(t)\in  F[t;\sigma]$ then ${\rm Aut}_F(S_f)=\{H_{\tau , 1}\}\cong {\rm Gal}(K/F)$.
\\ (ii) follows from (i) and Theorem \ref{Aut(S_f) subgroup corollaryII}.
\end{proof}

Note that indeed Condition (\ref{eqn:neccessary}) heavily restricts the choice of available
$k$ to $k=1$ in most cases.

\begin{corollary} \label{cor:Sandler}
Suppose $\sigma$ commutes with all $\tau\in {\rm Gal}(K/F)$. Let  $n
\geq m-1$ and  $f(t)=t^m-a_0\in K[t;\sigma]$, $a_0\in K\setminus F$.
\\ (i) $H\in {\rm Aut}_F(S_f)$ if and only if $H=H_{\tau , k}$ where   $k \in K^{\times}$ is such that
$ \tau(a_0) = \Big( \prod_{l=0}^{m-1}\sigma^l(k) \Big) a_0. $ In
particular, here $N_{K/F}(k)$ is an $m$th root of unity.
\\ (ii) For all $g(t) = t^m - \sum_{i=0}^{m-1} a_i t^i \in K[t; \sigma]$ with
$a_0\in K\setminus F$,  ${\rm Aut}_F(S_g)$ is a subgroup of
${\rm Aut}_F(S_f)$.
\end{corollary}

\begin{proof}
(i) follows from Theorem \ref{thm:automorphism_of_Sf_field_case} and
Proposition \ref{prop:conditionsonk}.
\\ (ii) follows from Theorem \ref{Aut(S_f) subgroup corollary}.
\end{proof}

For $f(t) = t^m-a_0 \in K[t;\sigma]$, $a_0\in K\setminus F$,
 the automorphisms $H_{id,k}$ extending $id_K$ thus  are in one-to-one correspondence with those $k$
 satisfying
$ \prod_{l=0}^{m-1}\sigma^l(k) =1$
(in particular, we have $N_{K/F}(k)^m=1$).
Analogously, we still obtain for $n < m-1$ employing Theorem
\ref{thm:automorphism_of_Sf_field_caseII} and Theorem \ref{Aut(S_f) subgroup corollaryII}:

\begin{corollary} \label{cor:SandlerII}
Suppose $\sigma$ commutes with all $\tau\in {\rm Gal}(K/F)$. Let  $n < m-1$
 and  $f(t)=t^m-a_0\in K[t;\sigma]$, $a_0\in K\setminus F$.
\\ (i) For all  $k \in K^{\times}$ with $N_{K/F}(k)$ an $m$th root of unity and
$ \tau(a_0) = \Big( \prod_{l=0}^{m-1}\sigma^l(k) \Big) a_0,$
 the maps  $H_{\tau , k}$ are automorphisms of $S_f$.
\\ (ii) For all $g(t) = t^m - \sum_{i=0}^{m-1} a_i t^i \in K[t; \sigma]$ with $a_0\in K\setminus F$,
$\{H\in {\rm Aut}_F(S_g)\,|\, H=H_{\tau , k} \}$ is a subgroup
of $\{H\in {\rm Aut}_F(S_f)\,|\, H=H_{\tau , k} \}.$
\end{corollary}

For $m=n$ and $K/F$ a cyclic field extension, the algebras considered in Corollary \ref{cor:Sandler} are called
\emph{nonassociative cyclic algebras of degree $m$}, as they can be
seen as canonical generalizations of associative cyclic algebras.
These algebras are treated in Section \ref{sec:nonasscyclic}.

%
%

\section{Automorphisms extending $id_K$ when $K/F$ is a cyclic field extension}\label{subsec:I}

 Let
$f(t) = t^m -\sum_{i=0}^{m-1} a_i t^i \in K[t; \sigma]$
not be invariant.
In general, we know that if $S_f$ has nucleus $K$ then every inner automorphism $G_c$ with $c\in K^\times$,
extends $id_K$. Conversely, an extension $H_{id,k}$ of $id_K$ is inner for the right choice of $k$:

\begin{lemma} \label{le:inner_general}
Let $k = c^{-1}\sigma(c)$ with $c \in K^{\times}$, then $H_{id,k}\in {\rm Aut}_F(S_f)$ is an inner automorphism.
\end{lemma}

\begin{proof}
 A simple calculation shows that
 $G_c \Big( \sum_{i=0}^{m-1} x_i t^i \Big) = \Big( c^{-1} \sum_{i=0}^{m-1} x_i t^i  \Big) c=
 x_0 +
 \\
 \sum_{i=1}^{m-1} x_i c^{-1} \sigma^i(c) t^i=H_{id, k}\Big( \sum_{i=0}^{m-1} x_i t^i \Big).$
 \end{proof}

Let now $K/F$ be a cyclic Galois field extension of degree $n$ with ${\rm Gal}(K/F) = \langle\sigma \rangle$ and norm
 $N_{K/F} : K \rightarrow F, \quad N_{K/F}(k) = k \sigma(k) \sigma^2(k) \cdots\sigma^{n-1}(k).$
 By Hilbert's Theorem 90,
 ${\rm ker}(N_{K/F})=\Delta^{\sigma}(1),$
 where $\Delta^{\sigma}(l)=\{\sigma(c)lc^{-1}\,|\, c\in K^\times\}$
is the $\sigma$-conjugacy class of $l\in K^\times$ \cite{LL}.

\begin{theorem} \label{thm:inner_general}
(i) Every automorphism $H_{id,k}\in {\rm Aut}_F(S_f)$ such that $N_{K/F}(k) = 1$ is an inner automorphism.
\\ (ii)
If  $n\geq m-1$ and $a_{m-1}\not=0$, or if
$n=m$, $ a_i=0 $ for all $i\not=0$ and $ a_0\in K\setminus F$, then
 these are all the automorphisms extending $id_K$.
\end{theorem}

\begin{proof}
 (i) Suppose there is $H_{id, k}\in {\rm Aut}_F(S_f)$ with $N_{K/F}(k) = 1$, then by Hilbert 90,
there exists $c \in K^{\times}$ such that $k = c^{-1}\sigma(c)$.
 Thus $H_{id,k}= H_{id, c^{-1}\sigma(c)}$ for  $c \in K^{\times}$ and so $G_c=H_{id, k}$ by Lemma \ref{le:inner_general}.
\\ (ii) By Theorem \ref{thm:automorphism_of_Sf_field_case} and Proposition \ref{prop:conditionsonk} (i), these are all the automorphisms extending $id_K$ when
 $n\geq m-1$ if $a_{m-1}\not=0$. The remaining assertion is proved analogoulsy.
\end{proof}

%
%

\section{Cyclic subgroups of ${\rm Aut}_F(S_f)$}\label{sec:cyclic_subgroups}

For any Galois field extension $K/F$ and $\sigma\in {\rm Gal}(K/F)$ of order $n$,
 we now give some conditions for ${\rm Aut}_F(S_f)$ to have  cyclic subgroups.

\begin{theorem} \label{thm:cyclic_subgrp}
Suppose $F$ contains an $s$th root of unity $\omega$.
Suppose that either
 $f(t) = t^s - a \in K[t;\sigma]$ where $a \in K \setminus F$, or
   $f(t) = t^{sl} -\sum_{i=0}^{l-1} a_{is} t^{is} \in K[t;\sigma]$ such that $S_f$ is not associative.
 Then $\langle H_{id,\omega}\rangle$ is a cyclic subgroup of ${\rm Aut}_F(S_f)$ of order at most $s$
 and of order $s$, if $\omega$ is a primitive
 root of unity.
\end{theorem}

\begin{proof}
(i) Let $f(t) = t^s - a$. Then
$\omega^j \sigma(\omega^j) \cdots \sigma^{s-1}(\omega^j) = \omega^{js} = 1$
and so  $H_{id,\omega^j}\in {\rm
Aut}_F(S_f)$ for all $0 \leq j \leq s-1$ by Proposition \ref{prop:automorphism_of_Sf_field_general}.
 \\ (ii) Let $f(t) = t^{sl} -\sum_{i=0}^{l-1} a_{is} t^{is}$. Then we have
$\prod_{q=is}^{ls-1} \sigma^q(\omega^j) = \omega^{j(ls-is)} = 1$
for all $i = 0,\ldots,l-1$. Hence $a_{is} = \Big(\prod_{q=is}^{ls-1} \sigma^q(\omega^j) \Big) a_{is}$ for all
$i=0,\ldots,l-1$ and so $H_{id,\omega^j}\in {\rm Aut}_F(S_f)$ for all $0 \leq j \leq s-1$
by Proposition \ref{prop:automorphism_of_Sf_field_general}.
\\
In both (i) and (ii), $\langle H_{id,\omega}\rangle$ is a cyclic subgroup of ${\rm Aut}(S_f)$ of order less or equal to $s$,
 since $H_{id,\omega^j} \circ H_{id,\omega^r} = H_{id,\omega^{j+r}}$ for all $0 \leq j,r \leq sl-1$.
\end{proof}

\begin{lemma}
Let $F$ have characteristic not two, $m$ be even and
$f(t) = t^m - \sum_{i=0}^{(m-2)/2}a_{2i} t^{2i}\in K[t;\sigma]$
not invariant. Then $\left\{ H_{id,1}, H_{id,-1}
\right\}$ is a subgroup of $S_f$ of order $2$.
\end{lemma}

\begin{proof}
The maps $H_{id,1}$ and $H_{id,-1}$ are automorphisms of $S_f$
 by Proposition \ref{prop:automorphism_of_Sf_field_general},  and $H_{id,-1} \circ H_{id,-1} = H_{id,1}$.
\end{proof}

If $f\in F[t]\subset K[t;\sigma]$, we obtain:

\begin{theorem} \label{thm:fixedfield}
Suppose $\sigma$
commutes with all  $\tau\in {\rm Gal}(K/F)$, and
 $f(t) = t^m - \sum_{i=0}^{m-1} a_it^i \in F[t;\sigma]\subset K[t;\sigma]$
 is not invariant.
\\ (i) $\langle H_{\sigma,1}\rangle \cong \mathbb{Z}/n \mathbb{Z}$ is a cyclic subgroup of ${\rm Aut}_F(S_f)$.
\\ (ii) Suppose ${\rm Gal}(K/F)=\langle\sigma\rangle$, $n=m$ is prime,  $a_0 \neq 0$ and not all of
$a_1, \ldots, a_{m-1}$ are zero. Then ${\rm Aut}_F(S_f) =
\langle H_{\sigma,1}\rangle\cong \mathbb{Z}/m \mathbb{Z}.$
\end{theorem}

\begin{proof}
Let $j \in \left\{ 0, \ldots, n-1 \right\}$. Since $\tau(a_i) = a_i$
for all $i$, here  (1)  becomes
\begin{equation} \label{eqn:Aut(S_f) cyclic, Coeffiecients in F, cyclic field extension case 1}
 a_i = \Big( \prod_{q=i}^{m-1} \sigma^q(k) \Big) a_i
\end{equation}
for all $i \in \left\{ 0, \ldots, m-1 \right\}$.
\\ (i) Clearly, (\ref{eqn:Aut(S_f) cyclic, Coeffiecients in F, cyclic field
extension case 1}) is satisfied for $k=1$ and all $i\in \left\{ 0,
\ldots, m-1 \right\}$, therefore the maps $H_{\tau,1}$ are
automorphisms of $S_f$ for all $\tau \in {\rm Gal}(K/F)$ by Theorems
\ref{thm:automorphism_of_Sf_field_case} and \ref{thm:automorphism_of_Sf_field_caseII}. We have
$H_{\sigma^j,1} \circ H_{\sigma^l,1} =H_{\sigma^{j+l},1}$
and  $ H_{\sigma^n,1}=H_{id,1} $. Hence $\langle H_{\sigma,1} \rangle = \left\{
H_{id,1}, H_{\sigma,1}, \ldots, H_{\sigma^{m-1}, 1} \right\}$ is a
cyclic subgroup of order $n$.
\\ (ii) By Theorem \ref{thm:automorphism_of_Sf_field_case}, the automorphisms
of $S_f$ are exactly the maps $H_{\sigma^j,k}$ where $j \in \left\{
0, \ldots, n-1 \right\}$ and $k \in K^{\times}$ satisfies
(\ref{eqn:Aut(S_f) cyclic, Coeffiecients in F, cyclic field extension
case 1}) for all $i \in \left\{ 0, \ldots, m-1 \right\}$. The maps
$H_{\sigma^j,1}$ are therefore automorphisms of $S_f$ for all $j \in
\left\{ 0, \ldots, n-1 \right\}$. We prove that these are the only
automorphisms of $S_f$: $a_0 \neq 0$ and so $N_{K/F}(k) = 1$ by
\eqref{eqn:Aut(S_f) cyclic, Coeffiecients in F, cyclic field
extension case 1}. Therefore, by Hilbert 90, there exists $\alpha \in
K$ such that $k = \sigma(\alpha)/\alpha$. Let $l \in \left\{1, \ldots, m-1 \right\}$ be such that $a_l \neq 0$. Then by
\eqref{eqn:Aut(S_f) cyclic, Coeffiecients in F, cyclic field extension case 1}, $$1 = \prod_{q=l}^{m-1} \sigma^q(k) =
\prod_{q=l}^{m-1} \sigma^q \big( \frac{\sigma(\alpha)}{\alpha} \big)= \frac{\prod_{q=l+1}^{m} \sigma^q(\alpha)}{\prod_{q=l}^{m-1}
\sigma^q(\alpha)} = \frac{\alpha}{\sigma^l(\alpha)}.$$ Thus $\alpha
\in \text{Fix}(\sigma^j) = F$ since $m$ is prime. Therefore
$k =\sigma(\alpha)/\alpha = \alpha/\alpha = 1$
as required.
\end{proof}

This complements our results from Theorem \ref{thm:automorphism_of_Sf_field_caseIV}, which
in case ${\rm Gal}(K/F)$ is cyclic of degree $n$ mean the following:

\begin{corollary}
Suppose ${\rm Gal}(K/F)$ is cyclic of degree $n$ and
$f(t) = t^m - \sum_{i=0}^{m-1} a_i t^i \in F[t;\sigma]$   not invariant with $a_{m-1}\in F^\times$.
\\ (i) Let  $n\geq m-1$ then for all
$\tau\in {\rm Gal}(K/F) $ the maps  $H_{\tau , 1}$  are exactly the automorphisms of $S_f$ and
${\rm Aut}_F(S_f)\cong {\rm Gal}(K/F)\cong \mathbb{Z}/n \mathbb{Z}.$
 \\ (ii) Let  $n < m-1$ then for all
$\tau\in {\rm Gal}(K/F) $ the maps  $H_{\tau , 1}$  are automorphisms of $S_f$ and
 ${\rm Gal}(K/F)\cong \mathbb{Z}/n \mathbb{Z}$ is isomorphic to a subgroup of  ${\rm Aut}_F(S_f)$.
 \end{corollary}

%
%

\section{Nonassociative  cyclic algebras} \label{sec:nonasscyclic}

\subsection{} Let $K/F$ be a cyclic Galois extension of degree $m$ with ${\rm
Gal}(K/F)=\langle\sigma\rangle$ and  $f(t)=t^m-a\in K[t;\sigma]$.
Then $(K/F,\sigma,a)=K[t;\sigma]/K[t;\sigma](t^m-a)$
 is called a \emph{nonassociative cyclic algebra of degree $m$} over $F$. It is not associative for all $a\in K\setminus F$
 and a cyclic associative central simple algebra over $F$ for $a\in F^\times$.
 We will only consider the case that $a\in K\setminus F$. If  $1,a,a^2, \ldots, a^{m-1}$ are linearly
independent over $F$ then $(K/F, \sigma, a)$ is a division algebra (cf. \cite{S12},  \cite{San62} for finite $F$). In particular, if
$K/F$ is of prime degree then $(K/F, \sigma, a)$ is a division algebra for every  $a \in K\setminus F$.

\begin{theorem} \label{thm:cyclic_subgroupsII}
Let $A=(K/F, \sigma, a)$ be a nonassociative cyclic algebra of degree
$m$.
\\ (i) All the automorphisms of $A$ which extend $id_K$ are inner automorphisms
and of the form $H_{id,l}$ for all $l\in K^\times$ such that
$N_{K/F}(l)=1$.
\\ The subgroup they generate in ${\rm Aut}_F(A)$ is isomorphic to ${\rm ker}(N_{K/F})$.
\\ (ii) An automorphism $\sigma^j\not=id$ can be extended to  $H\in {\rm Aut}_F(A)$, if and only if there is some
$l\in K$ such that
 $\sigma^j(a)= N_{K/F}(l)a.$
  In that case, $H=H_{\sigma^j,l}$ and if $m$ is prime then $N_{K/F}(l)=\omega$ for an $m$th root of unity
  $1\not=\omega\in F$.
\\ (iii)
 Let $c \in K \setminus F$ and suppose there exists $r \in\mathbb{N}$ such that $c^r \in F^\times$.
Let $r$ be minimal. Then $\langle G_c \rangle$ is a cyclic subgroup of ${\rm
Aut}_F(S_f)$ of order $r$.
\end{theorem}

\begin{proof}
Theorem \ref{thm:automorphism_of_Sf_field_case},  Theorem
\ref{thm:inner_general} (ii),  and Proposition \ref{prop:conditionsonk}
imply (i) and (ii).
\\ (iii) Let $c \in K \setminus F$. Then $G_c$ is an automorphism, because $K$ is the nucleus of $A$.
 Since
$G_c \circ G_c = G_{c^2}$, $G_c \circ G_c \circ G_c= G_{c^3}$ and so
on, we have $G_{c^r} = id$ if and only if $c^r \in F$. If $r \in
\mathbb{N}$ is smallest possible
 then $\langle G_c\rangle$ is a cyclic subgroup of $ {\rm Aut}_F(S_f)$ of order $r$.
\end{proof}

Note that different roots of unity yield different $l$ in Theorem \ref{thm:cyclic_subgroupsII} (ii). This yields:

\begin{theorem} \label{thm:aut1}
Let  $A=(K/F,\sigma,a)$ be a nonassociative cyclic algebra of degree
$m$.
Suppose  $F$ contains a non-trivial $m$th root of unity $\omega$.
\\ (i) $\langle H_{id,\omega}\rangle$
is a cyclic subgroup of ${\rm Aut}_F(A)$ of order at most $m$.
 If $\omega$ is a primitive $m$th root of unity, then $\langle H_{id,\omega}\rangle$
has order $m$.
\\ (ii) If there is an element $l\in K$, such that $N_{K/F}(l)=\omega$ for  $\omega$  a primitive $m$th root of unity and
$\sigma(d)= \omega d$, then the subgroup generated by $H_{\sigma,l}$
has order $m^2$.
\end{theorem}

\begin{proof}
 (i) follows from Theorem \ref{thm:cyclic_subgrp}.
\\ (ii) Suppose $\sigma$ can be extended to an $F$-automorphism $H$ of $A$. Then by
Theorem \ref{thm:cyclic_subgroupsII}, there is an element $l\in K$,
such that $N_{K/F}(l)=\omega$, $\omega\not=1$ and $\sigma(d)= \omega
d$, and $H=H_{\sigma,l}$.
(If $1=N_{K/F}(l)$, then $\sigma(d)= d$, contradiction.)

The subgroup generated by $H=H_{\sigma,l}$ has order greater than $m$, since
$H_{\sigma,l}\circ \cdots\circ H_{\sigma,l}$ ($m$-times) becomes
$H_{\sigma^{m},b}=H_{id,\omega}$ with $\omega=N_{K/F}(l)$. $H_{id,\omega}$ has order $m$, so
 the subgroup generated by $H=H_{\sigma,l}$ has order $m^2$.
\end{proof}

\subsection{The case that $m$ is prime}\label{subsec:nasc}

 Let us now assume that the cyclic field extension  $K/F$ has prime degree $m={\rm deg} (f)$.
Suppose that $F$ contains a primitive $m$th root of unity, where $m$ is prime to the characteristic of $F$.
Then $K = F(d)$, where $d$ is a root of an irreducible polynomial $t^m - c \in F[t]$.

\begin{lemma}\label{egnval} (cf. \cite[Lemma 6.2.7]{AndrewPhD})
 The eigenvalues of $\sigma^j\in {\rm Gal}(K/F)$ are precisely the $m$th roots of unity. Moreover, the only
possible eigenvectors are of the form $ed^i$ for
some $i$, $0 \leq i \leq m-1$ and some $e\in F$.
\end{lemma}

Let $f(t) = t^m - a \in K[t; \sigma], \quad a\not\in F.$
Then we get the following strong restriction for automorphisms of $S_f$:

\begin{theorem}  \label{thm:automorphism_of_Sf_field_caseIII}
  $H$ is an
automorphism of $S_f$ extending $\sigma^j\not=id$ if and only if
$H=H_{\sigma^j , k}$ for some $k \in K^{\times}$,
 where
$N_{K/F}(k)$
 is an $m$th root of unity and
$a=ed^s$ for some $e\in F^\times$ and some $d^{s}$.
\end{theorem}

\begin{proof}
$H$ is an automorphism of $S_f$ if and only if $H=H_{\sigma^j ,k}$ where $j \in \left\{ 0, \ldots, m-1 \right\}$ and $k \in
K^{\times}$ is such that
$ \sigma^j(a) = \Big(\prod_{l=0}^{m-1}\sigma^l(k) \Big) a=N_{K/F}(k)a .$
 For all $\sigma^j\not=id$, by Lemma \ref{egnval} this condition is equivalent to
$N_{K/F}(k)$ being an $m$th root of unity and $a=ed^{s}$ for some $d^{s}$
and $e\in F^\times$, for all  $k\in K^{\times}$.
\end{proof}

Applying Theorem \ref{Aut(S_f) subgroup corollary}, our results for the automorphisms of a
 nonassociative cyclic algebra $A=(K/F,\sigma,a)$ of degree $m$
  yield the following observations for more general algebras $S_g$:

\begin{corollary}
Suppose ${\rm Gal}(K/F)=\langle\sigma\rangle$ is cyclic of degree $m$ and
$g(t) = t^m - \sum_{i=0}^{m-1} a_i t^i \in K[t;\sigma]$  is not invariant with
$a_{0}\in K\setminus F$.
Suppose one of the following holds:
\begin{itemize}
\item $F$ has no $m$th root of unity.
\item $m$ is prime and $F$ contains a primitive $m$th root of unity, where $m$ is
prime to the characteristic of $F$. Let $K = F(d)$ as in  Section
\ref{subsec:nasc} and $a_0\not=e d^i$, $e\in F^\times$.
\end{itemize}
Then every $F$-automorphism of $S_g$ leaves $K$ fixed, is inner and
${\rm Aut}_F(S_g)$ is a subgroup of ${\rm ker}(N_{K/F})$, thus cyclic.
In particular, if ${\rm ker}(N_{K/F})$ has prime order, then either
${\rm Aut}_F(S_g)$ is trivial or ${\rm Aut}_F(S_g)\cong {\rm ker}(N_{K/F})$.
\end{corollary}


\subsection{The automorphism groups of nonassociative quaternion algebras}


Recall  the \emph{dicyclic group}
\begin{equation} \label{eqn:dicyclic group presentation}
{\rm Dic}_l = \langle x, y \ \vert \ y^{2l}=1, \ x^2 = y^l, \ x^{-1}yx = y^{-1} \rangle
\end{equation}
 of order $4l$.
The semidirect product $\mathbb{Z} / s \mathbb{Z} \rtimes_l \mathbb{Z} / n \mathbb{Z}$ between the cyclic groups
$\mathbb{Z} / s \mathbb{Z}$ and
$\mathbb{Z} / n \mathbb{Z}$ corresponds
to a choice of an integer $l$ such that $l^n \equiv 1 \ {\rm  mod} \ s$. It can be described by the presentation
$\mathbb{Z} / s \mathbb{Z} \rtimes_l \mathbb{Z} / n \mathbb{Z} =
 \langle x,y \ \vert \ x^s = 1, \ y^n= 1, \ yxy^{-1} = x^l  \rangle.$

We obtain the following result for the automorphism groups of
 nonassociative quaternion algebras (where $m=2$):

\begin{theorem} \label{thm:semidirect and dicyclic m=2}
Suppose $K = F(\sqrt{b})$ is a quadratic field extension of $F$,  $ char(F) \neq 2$, and consider the nonassociative
quaternion algebra $A=(K/F, \sigma, \lambda \sqrt{b})$ for some $\lambda \in F^{\times}$.
 Suppose there exists $k \in K^{\times}$ such that $k \sigma(k) = -1$.

 For every $c \in K \setminus F$ for which there is
a positive integer $j$ such that $c^j \in F^{\times}$, pick the smallest such $j$.
\\ (i) If $j$ is even
then ${\rm Aut}_F(S_f)$ contains the dicyclic group of order $2j$.
\\(ii)
 If $j$ is odd then ${\rm Aut}_F(S_f)$ contains a subgroup isomorphic to the semidirect product
$\mathbb{Z} / j \mathbb{Z} \rtimes_{j-1} \mathbb{Z} / 4 \mathbb{Z}.$
 In particular, ${\rm Aut}_F(A)$ always contains a subgroup isomorphic to
$\mathbb{Z}/4\mathbb{Z}$.
\end{theorem}

\begin{proof}
Since $\sigma(\sqrt{b}) = -\sqrt{b}$  and $k \sigma(k) = -1$, $H_{\sigma,k} \in {\rm Aut}_F(S_f)$
by Theorem \ref{thm:cyclic_subgroupsII}.
 A simple calculation shows that
$\langle H_{\sigma,k}\rangle = \{ H_{\sigma,k}, H_{id,-1}, H_{\sigma,-k}, H_{id, 1} \}.$
 $\langle G_c \rangle$ is a cyclic subgroup of ${\rm Aut}_F(S_f)$ of order $j$
 by Theorem \ref{thm:cyclic_subgroupsII} (iii).
\\ (i) Suppose $j$ is even and write $j = 2l$. We prove first that $G_{c^l} = H_{id, -1}$. Write
$c^l = \mu_0 + \mu_1 \sqrt{b}$ for some $\mu_0, \mu_1 \in F$. Then
$c^j = c^{2l} = \mu_0^2 + \mu_1^2 b + 2 \mu_0 \mu_1 \sqrt{b} \in F$
which implies $2 \mu_0 \mu_1 = 0$. Hence $\mu_0 = 0$ or $\mu_1 = 0$. Since $j$ is minimal,
$c^l \notin F$ so $\mu_0 = 0$ and $c^l = \mu_1 \sqrt{b}$. We obtain
\begin{align*}
G_{c^l}(x_0 + x_1t) &= x_0 + x_1 (\mu_1 \sqrt{b})^{-1} \sigma(\mu_1 \sqrt{b})t \\
&= x_0 + x_1 \mu_1^{-1} b^{-1} \sqrt{b}(- \mu_1 \sqrt{b})t \\
&= x_0 - x_1t = H_{id, -1}(x_0 + x_1t)
\end{align*}
which implies $G_{c^l} = H_{id, -1}$. Next we prove $(H_{\sigma,k})^{-1} G_c H_{\sigma,k} = G_c^{-1}$. Simple calculations show $(H_{\sigma,k})^{-1} = H_{\sigma,-k}$ and $G_c^{-1} = G_{\sigma(c)}$. We have
\begin{align*}
H_{\sigma,-k} \big( G_c \big( H_{\sigma,k} \big( x_0 + x_1t \big) \big) \big) &= H_{\sigma,-k} \big( G_c \big( \sigma(x_0) + \sigma(x_1)kt \big) \big) \\
&= H_{\sigma,-k} \big( \sigma(x_0) + \sigma(x_1)k c^{-1}\sigma(c) t \big) \\
&= x_0 - x_1 \sigma(k) \sigma(c^{-1})ckt \\
&= x_0 + x_1 \sigma(c^{-1})c t = G_{\sigma(c)}\big( x_0 + x_1t \big)
\end{align*}
and so $(H_{\sigma,k})^{-1} G_c H_{\sigma,k} = G_c^{-1}$.

Thus $H_{\sigma,k}^2 = H_{id, -1} = G_{c^l}=G_c^l, \ G_c^{2l} = id$ and
$(H_{\sigma,k})^{-1} G_c H_{\sigma,k} = G_c^{-1}$. Hence $\langle H_{\sigma,k}, G_c \rangle$ has the presentation
 \eqref{eqn:dicyclic group presentation} as required.
\\ (ii) Suppose $j$ is odd. Then $\langle G_c \rangle$ does not contain $H_{id, -1}$ as $H_{id, -1}$ has order $2$
 which implies $\langle H_{\sigma,k}\rangle \cap \langle G_c\rangle = \{ id \}$. Furthermore $(H_{\sigma,k})^{-1} G_c H_{\sigma,k} = G_c^{-1}
 = G_c^{j-1} = G_{c^{j-1}}$ can be shown similarly as in (i). Note that
$(j-1)^4 = j^4 - 4j^3 + 6j^2 - 4j + 1 \equiv 1 \text{ mod }(j).$
Thus ${\rm Aut}_F(S_f)$ contains the subgroup
$\langle G_c\rangle \rtimes_{j-1} \langle H_{\sigma,k}\rangle \cong \mathbb{Z} / j \mathbb{Z} \rtimes_{j-1}
 \mathbb{Z} / 4 \mathbb{Z}$
as required.
\\
In particular, choose $c = \sqrt{b}$ in (i), so that $j=2$. This implies ${\rm Aut}_F(A)$
contains the dicyclic group of order $4$,
which is the cyclic group of order $4$.
\end{proof}

\begin{example}
 (i) Let $F = \mathbb{Q}(i)$, $K = F(\sqrt{-3})$, $\sigma(\sqrt{-3})=-\sqrt{-3}$ and
 $A=(K/F,\sigma, \lambda \sqrt{-3})$ be a nonassociative quaternion algebra with
 some $\lambda \in F^{\times}$. Note that for $k=i$ we have $i \sigma(i)  = -1$.
 Let $c = 1 + \sqrt{-3}$. Then
$c^2 = -2+2\sqrt{-3} \text{ and } c^3 = -8$
which implies $j=3$ here.
Therefore ${\rm Aut}_F(S_f)$ contains a subgroup isomorphic to the semidirect product
$\mathbb{Z} / 3\mathbb{Z} \rtimes_2 \mathbb{Z} / 4 \mathbb{Z}$
by Theorem \ref{thm:semidirect and dicyclic m=2}.
\\ (ii) Let $F = \mathbb{Q}(i)$, $K = F(\sqrt{-1/12})$,
$\sigma(\sqrt{-1/12})=-\sqrt{-1/12}$ and $A=(K/F,\sigma, \lambda \sqrt{-1/12})$ be a nonassociative quaternion algebra
 for some
$\lambda \in F^{\times}$. Again for $k=i$ we have $i \sigma(i)  = -1$. Let $c = 1 + 2 \sqrt{-1/12}$. Then
$c^2 = 2/3 + 2i/\sqrt{3},$ $ c^3 = 8i/3\sqrt{3},$ $c^4 = -8/9 + 8i/3 \sqrt{3},$
 $ c^5 = -16/9 + 16i/9\sqrt{3}$ and $c^6 = -64/27.$
Hence $c, c^2, c^3, c^4, c^5 \in K \setminus F$ and $c^6 \in F$. Therefore ${\rm Aut}_F(A)$ contains the dicyclic
group of order $12$ by Theorem \ref{thm:semidirect and dicyclic m=2}.

\end{example}

%
%

\section{Isomorphisms between $S_f$ and $S_g$} \label{sec:iso}

The proofs of the previous sections can be adapted to check when two Petit algebras are isomorphic and when not.
This is not the main focus of this paper so we just point out how some of the results can be transferred.

If $K$ and $L$ are fields,   and $S_f=K[t;\sigma]/K[t;\sigma]f(t)\cong
L[t;\sigma']/L[t;\sigma']g(t)=S_g,$
 then $K\cong L,$ $ Nuc_r(S_f)\cong Nuc_r(S_g),$ $ {\rm deg}(f)={\rm deg}(g),$ and
${\rm Fix}(\sigma)\cong {\rm Fix}(\sigma'),$
since isomorphic algebras have the same dimensions, and isomorphic nuclei and center.

If $G$
is an automorphism of $R =K[t;\sigma]$ which restricts to an automorphism $\tau$ on $K$ which commutes with
$\sigma$, $f \in R$ is irreducible and $g(t) = G(f(t))$, then $G$ induces an isomorphism $S_f \cong
S_{G(g)}$ \cite[Theorem 7]{LS} (the proof works for any base field).

From now on let $F$ be the fixed field of $\sigma$,
$\sigma$ have
order $n$, and both
 $f(t) = t^m - \sum_{i=0}^{m-1} a_i t^i $ and $ g(t) =t^m - \sum_{i=0}^{m-1} b_i t^i \in K[t;\sigma]$
 be not invariant.
Then the following is proved analogously
to Theorem \ref{thm:automorphism_of_Sf_field_case}, Theorem
\ref{thm:automorphism_of_Sf_field_caseII} and Proposition \ref{prop:automorphism_of_Sf_field_general}:

\begin{theorem} \label{general_isomorphism_theorem}
Suppose $\sigma$ commutes with all $\tau\in {\rm Aut}_F(K)$ and $n \geq m-1$. Then $S_f \cong S_g$ if and only if there exists
$\tau\in {\rm Aut}_F(K)$ and $k \in K^{\times}$ such that
\begin{equation} \label{equ:necII}
\tau(a_i) = \Big( \prod_{l=i}^{m-1} \sigma^l(k) \Big) b_i
\end{equation}
 for
all $i \in\{ 0, \ldots, m-1\}$. Every such $\tau$ and $k$ yield a
unique isomorphism $G_{\tau,k}: S_f \rightarrow S_g$,
 $$G_{\tau,k}( \sum_{i=0}^{m-1}x_i t^i )= \tau(x_0) + \sum_{i=1}^{m-1}
\tau(x_i) \prod_{l=0}^{i-1}\sigma^l(k)  t^i.$$
\end{theorem}

If $n<m-1$ we still get a partial result:

\begin{theorem} \label{general_isomorphism_theoremII}
Suppose there exists $\tau \in {\rm Aut}_F(K)$ and $k \in K^{\times}$
such that $\tau\circ\sigma= \sigma\circ\tau$ and such that (\ref{equ:necII}) holds
for all $i \in \{ 0, \ldots, m-1 \}$. Then $S_f \cong S_g$ with an isomorphism
given by
$$G_{\tau,k}( \sum_{i=0}^{m-1}x_i t^i )=\tau(x_0) + \sum_{i=1}^{m-1}
\tau(x_i) \prod_{l=0}^{i-1}\sigma^l(k) t^i$$
 as in Theorem \ref{general_isomorphism_theorem}.
\end{theorem}

\begin{corollary} \label{prop:isomorphism_of_Sf_field_general}
  For every $k \in K^{\times}$ such that $a_i = \Big( \prod_{l=i}^{m-1} \sigma^l(k) \Big) b_i$ for all
 $i \in \{ 0, \ldots, m-1 \}$,  $G_{id,k}: S_f\rightarrow S_g$ is an isomorphism.
\end{corollary}

As a direct consequence of Theorem \ref{general_isomorphism_theorem}
we obtain:

\begin{theorem} \label{thm:condition_f_g}
Suppose $\sigma$ commutes with all $\tau\in {\rm Aut}_F(K)$ and  $n \geq m-1$.
 If $S_f \cong S_g$, then $a_i=0$ if and only if $b_i=0$, for all $i \in\{ 0, \ldots, m-1\}$.
\end{theorem}

\begin{proof}
 If $S_f \cong S_g$ then by Theorem
\ref{general_isomorphism_theorem}, there exists $j \in \left\{
0,\ldots, n-1 \right\}$ and $k \in K^{\times}$ such that $\tau(a_i) =
\Big( \prod_{l=i}^{m-1} \sigma^l(k) \Big) b_i$
 for all $i \in\{ 0, \ldots, m-1\}$. This implies $a_i=0$ if and only if $b_i=0$, for all $i \in\{ 0, \ldots, m-1\}$.
\end{proof}

From now on we restrict ourselves to the situation that  $R=K[t;\sigma]$ and $F={\rm Fix}(\sigma)$,
 where $K/F$ is a finite Galois field extension and  $\sigma$ of order $n$.
We take a closer look at the consequences of Equality (\ref{equ:necII}):

\begin{proposition} \label{prop:conditionsonkII}
 Let  $k \in K^{\times}$  such that
$\tau(a_i) = \Big( \prod_{l=i}^{m-1}\sigma^l(k) \Big) b_i$ for all
$i \in\{ 0, \ldots, m-1\}$. Then $a_i=0$ if and only if $b_i=0$ and:
\\ (i) For all $i \in\{ 0, \ldots, m-1\}$ with $a_i\not=0$,
$N_{K/F}(a_i) =  N_{K/F}(k)^{m-i} N_{K/F}(b_i).$
\\ (ii) If there is some $i \in\{ 0, \ldots, m-1\}$ such that $a_i\in {\rm Fix}(\tau)^\times$, then
$a_i/b_i = \prod_{l=i}^{m-1}\sigma^l(k) .$
In particular, if $a_{m-1}\in F^\times$ and $b_{m-1}\in F^\times$, then
 $k\in F^\times$ and $a_i = k^{m-i} b_i$ for all $i \in\{ 0, \ldots, m-1\}$.
 \\ (iii) If  $a_0\in {\rm Fix}(\tau)^\times$, $m=n$ and ${\rm Gal}(K/F)=\langle\sigma\rangle$ then
 $a_0 =N_{K/F}(k)b_0.$
\end{proposition}

\begin{proof}
(i) Equality (\ref{equ:necII})
 implies that
$N_{K/F}(a_i) =  \prod_{l=i}^{m-1}N_{K/F}(\sigma^l(k))
N_{K/F}(b_i)$ (simply apply $N_{K/F}$ to both sides of (\ref{equ:necII})), therefore
$N_{K/F}(a_i) =  N_{K/F}(k)^{m-i} N_{K/F}(b_i)$ for all $i \in\{ 0,
\ldots, m-1\}$ is a necessary condition on $k$.
\\ (ii)  If there is an $i \in\{ 0, \ldots, m-1\}$ such that $a_i\in {\rm Fix}(\tau)^\times$, then
 (\ref{equ:necII})
 implies that $a_i = \Big( \prod_{l=i}^{m-1}\sigma^l(k) \Big) b_i$, so that we obtain
$a_i/b_i = \prod_{l=i}^{m-1}\sigma^l(k). $

Alternatively, if $a_{m-1}\in F^\times$ and  $b_{m-1}\in F^\times$, then $a_{m-1} = \sigma^{m-1}(k)b_{m-1} $
imply $k\in F^\times$, hence $a_i = k^{m-1-i} b_i$ for all
$i \in\{ 0, \ldots, m-1\}$.
\\ (iii) In particular, if  $a_0\in
{\rm Fix}(\tau)^\times$, $m=n$ and $\sigma$ generates ${\rm Gal}(K/F)$, then
$a_0/b_0 = \prod_{l=0}^{m-1}\sigma^l(k)=N_{K/F}(k)$
 is a necessary condition on $k$.
\end{proof}

\begin{corollary}
Suppose $\sigma$ commutes with all $\tau\in {\rm Gal}(K/F)$ and $n \geq m-1$.
Assume that one of the following holds:
\\ (i)  There exists  $i \in\{ 0, \ldots, m-1\}$ such that $b_i\not=0$ and
$N_{K/F}(a_i b_i^{-1}) \not\in  N_{K/F}(K^\times)^{m-i};$
\\ (ii) $m=n$, $a_0\in F^\times$ and $b_0\in K\setminus F$.
\\
Then $S_f \ncong S_g$.
\end{corollary}

\begin{corollary} \label{Isomorphism cyclic extension}
Suppose  $ {\rm Gal}(K/F)=\langle \sigma\rangle$ and $n =m$.
Let $f(t) = t^m - a,$ $g(t) = t^m - b \in K[t;\sigma]$ where $a, b\in K \setminus F$.
\\ (i) $S_f \cong S_g$ if and only if
there exists $\tau\in {\rm Gal}(K/F)$ and $k \in K^{\times}$ such that
$\tau(a) = N_{K/F}(k) b.$
 (ii) If $\sigma^j(a) \neq N_{K/F}(k)b$
  for all $k \in K^{\times}$, $j = 0, \ldots, m-1$, then $S_f \not\cong S_g$.
\end{corollary}

These follow from
Proposition \ref{prop:conditionsonkII}.
Note that Corollary \ref{Isomorphism cyclic extension} canonically generalizes well-known criteria for associative
cyclic algebras.


\end{document}